\documentclass[a4paper,11pt]{article}
\usepackage[utf8x]{inputenc}
\usepackage[left=20mm, right=20mm, top=20mm, bottom=20mm]{geometry}
\usepackage{color, hyperref}
\usepackage{amsmath}
\usepackage{amssymb}
\usepackage{amsfonts}
\usepackage{amsthm,amscd}
\usepackage{authblk}
\newtheorem{theorem}{Theorem}[section]
\newtheorem{lemma}[theorem]{Lemma}

\newtheorem{proposition}[theorem]{Proposition}
\theoremstyle{definition}
\newtheorem{definition}[theorem]{Definition}

\newtheorem{que}{Question}
\newtheorem{conjecture}{Conjecture}

\theoremstyle{remark}
\newtheorem{remark}[theorem]{Remark}
\newtheorem{thmx}{{\bf Theorem}}

\newtheorem{corx}{{\bf Corollary}}

\numberwithin{equation}{section}

\def\R{\mathbb R}

\def\cal{\mathcal}

\title{On Banyaga's conjecture on  the group of symplectic homeomorphisms }

\author[1]{Carole Madengko \thanks{carolemadengko@gmail.com}}
\author[2]{Stephane Tchuiaga \thanks{tchuiagas@gmail.com}}
\author[3]{Franck Houenou\thanks{rdjeam@gmail.com}}
\affil[1,3]{Institut de Math\'ematiques et de Sciences Physiques,  University of Abomey Calavi, Abomey Calavi, Benin}
\affil[2]{Department of Mathematics, University of Buea, South West Region, Cameroon}

\date{ }

\begin{document}
	\maketitle

	\begin{abstract}
		
	This paper addresses Banyaga's conjecture  asserting  that : the group  of strong symplectic homeomorphisms is a proper normal subgroup of the symplectic homeomorphism group of a closed symplectic manifold.


	\end{abstract}	
	\section{Introduction}
On a closed symplectic manifold $(M,\omega)$, denote by $G_{\omega}(M)$  the identity component in the group $Symp(M,\omega)$ consisted of  symplectic diffeomorphisms. The group $G_{\omega}(M)$  carries a Hofer-like topology introduced  in \cite{ba10a} and  further examined  in \cite{Tc18}. Exploring   the $C^0$  and  Hofer-like  topologies  on $G_{\omega}(M)$,  led the authors in \cite{Ba08, Tch-al} to  define the  group $SG_{\omega}(M)$ of strong symplectic homeomorphisms as a $C^0-$counterpart of $G_{\omega}(M)$. In \cite{Ba10} Banyaga showed  that  $SG_{\omega}(M)$ is a subgroup of $Sympeo_{0}(M,\omega)$  the identity  component in the   group of symplectic homeomorphisms, and conjectured the the following
\begin{conjecture}\label{con}
The inclusion   $SG_{\omega}(M)\subset Sympeo_{0}(M,\omega)$ is proper.	
\end{conjecture} 
 When $M$ is simply connected, $SG_{\omega}(M)$ coincides with $Hameo(M,\omega)$, the hameomorphism group defined by Oh-M\"uller \cite{Oh-M07}.
 
As will be detailed in Section \ref{pre}  any symplectic isotopy  $\Phi=\{\phi^t\}$ is generated by a pair $(U,\mathcal{H})$, where $U=\{U^t\}$ is a smooth  time dependent function on $M\times [0,1]$ and $\mathcal{H}=\{\mathcal{H}^t\}$ is a smooth family of  $S-$forms or  harmonic forms. Therefore, the Hofer-like length of $\Phi=\{\phi^t\}$  is given by 

\begin{equation}
	l^{(1,\infty)}_{\kappa, \mathcal{S}}(\Phi) 
= \int_{0}^{1}\left ( osc(U^t) + \kappa\| \mathcal{H}^t\| _{L^2}\right )dt,
\end{equation} 
where $osc(\cdot)= \max(\cdot)-\min(\cdot)$,  $\kappa$ is a positive real number and $\|\cdot\|_{L^2}$ is the $L^2$-norm on $H^1(M,\R)$.
The length $l^{(1,\infty)}_{\kappa, \mathcal{S}}$ induces  metrics $D_{\kappa, \mathcal{S}}^{1}$ and $d_{HL}$, on  $Iso(M,\omega)$ the group of symplectic isotopies   and $G_{\omega}(M)$ respectively.
\begin{definition}
	A homeomorphism $\phi$ is called a finite symplectic  energy homeomorphism if there exists a sequence $(\Phi_{i})_{i}=(\{\phi_{i}^t\})_{i}$ of symplectic isotopies generated by $(U_{i},\mathcal{H}_{i})_{i}$ such that $\phi_{i}^{1}\xrightarrow{C^0}\phi$ and 	$l^{(1,\infty)}_{\kappa, \mathcal{S}}(\Phi_{i}) $ is bounded for all $i$.
\end{definition}
Denote  by $FSHomeo(M)$ the finite energy symplectic 
\begin{definition}\cite{Tch-al}\label{S1}
	A continuous  isotopy $\Phi$ of homeomorphisms   of $M$ 
	is called an $\mathcal{S}-$topological isotopy if 
	there exists a pair $(U,\mathcal{H})$ (with $U$  a continuous time dependent function on $M\times [0,1]$ and $\mathcal{H}$  a continuous family of  $S-$forms or  harmonic forms), and  a sequence of symplectic isotopies $\{\Phi_i\}_i$ generated by $((U_i,\mathcal{H}_i))_{i}$
	such that 
	 $\lim\limits_{i\longrightarrow\infty}D^{1}_{\kappa, \mathcal{S}}((U_i,\mathcal{H}_i),(U,\mathcal{H}))=0$ and  $\Phi_i\xrightarrow{C^0}\Phi$. 
\end{definition}

The group  $\mathcal{S}G_\omega(M)$  of  strong symplectic homeomorphisms of  $M$  consists of  time-one maps of  $\mathcal{S}-$topological isotopies.
 

If $M$ is simply connected, $FSHomeo(M)$ coincides with the finite energy hamiltonian homeomorphism group $FHomeo(M)$ defined in \cite{cr2}.  It has been shown in \cite{bu, cr2} that on any  closed simply connected surface $ (\Sigma,\omega)$,  the inclusions 
\begin{eqnarray}\label{in}
	\mathcal{S}G_\omega(\Sigma)\subseteq FSHomeo(\Sigma)\subseteq Sympeo_{0}(\Sigma,\omega),\end{eqnarray} are strict.
To our knowledge it is not known whether the  inclusions (\ref{in}) are still strict on closed surfaces with genus $g>0$ and higher dimensional closed symplectic manifolds.

In \cite{tc21}, Tchuiaga showed that function 
\begin{eqnarray}\label{flx}
	\begin{array}{cccl}
	\Delta(\cdot,\alpha) :& G_{\omega}(M)&\longrightarrow& \R\\
	& \phi^1& \longmapsto &	\Delta(\phi^1,\alpha):= \displaystyle \dfrac{1}{\|\alpha\|_{L^2}}\int_{M}\left (\int_{\gamma} 	(\phi^1)^*\alpha -\alpha\right )\frac{\omega^n}{n!}
\end{array}\nonumber
\end{eqnarray}
is continuous w.r.t the $C^0-$topology, hence admits a $C^0-$extension  say  $\chi(\cdot,\alpha): \overline{G_{\omega}(M)}^{C^0}\to \R$.
Since $\overline{G_{\omega}(M)}=Sympeo_{0}(M, \omega)$ on surfaces, we use the latter extension to proof the following result.  
\begin{thmx}\label{A1}
	On a closed symplectic surface $(\Sigma_{g},\omega)$ of genus $g\geqslant1$,	$FSHomeo(\Sigma_{g})$ is a proper normal subgroup of $Sympeo_{0}(\Sigma_{g},\omega).$
\end{thmx}
Theorem \ref{A1} confirms Conjecture \ref{con}  for surfaces of genus $g>0$, and gives an affirmative answer to the following question posed by Banyaga in \cite{Ba10} 
\begin{que}\label{Q}
	Is $\mathcal{S}G_\omega(M)$ a normal subgroup of  $Sympeo_{0}(M,\omega)$.
\end{que}
 The following result generalises Theorem $1.4-$\cite{bu}.

\begin{thmx}\label{B1}
		For every $E>0$, there exists a continuous path $\{\psi^t\}$ of homeomorphism of $M$ such that 
	\begin{enumerate}
		\item the path $\{\psi^t\}$ is a uniform limit of a sequence of smooth symplectic isotopies $(\Phi_{i})_{i}=(\{\phi^t_{i}\})_{i}$ generated by $(U_{i}, \mathcal{H}_{i})_{i}$ and $	l^{(1,\infty)}_{\kappa, \mathcal{S}}(\Phi_{i})\leqslant E$
		\item for every sequence $(\varphi_{i})_{i}\subset G_{\omega}(M)$, satisfying $\psi^1=\lim\limits_{C^0}\varphi_{i}$, and for every $\varphi\in G_{\omega}(M)$, we have  $ \liminf\limits_{i\longrightarrow\infty}d_{HL}(\varphi_{i},\varphi)\geqslant E$ 
	\end{enumerate}
\end{thmx}
Item  $1$ of Theorem \ref{B1} tells us that $\psi^1\in FSHomeo(M)$.
Any such $\psi^1$  can not be  in $\mathcal{S}G_{\omega}(M)$. Indeed, suppose $\psi\in\mathcal{S}G_{\omega}(M)$, there exists a sequence $(U_{i},\mathcal{H}_{i})_{i}\subseteq \mathfrak{T}(M, \omega, \mathcal{S})$  generating a sequence $(\{\phi_{i}^t\}_{t})_{i}\subseteq Iso(M,\omega)$ with $(U_{i},\mathcal{H}_{i})_{i}\xrightarrow{D^{1}_{\lambda, \mathcal{S}}} (U',\mathcal{H}')$ and $\phi_{i}^1\xrightarrow{C^0}\psi$. Fix $l$ sufficiently large, consider the pair $(U_{l},\mathcal{H}_{l})$ which generates $\phi_{l}^1\in G_{\omega}(M)$. Then for any  $\epsilon>0$, we have 
$\liminf\limits_{i\longrightarrow\infty}d_{HL}(\phi_{i}^1,\phi_{l}^1)\leqslant \epsilon$. This contradicts item $2$ of Theorem \ref{B1} hence $\psi^1\notin \mathcal{S}G_{\omega}(M)$.
\begin{corx}\label{bg}
	$\mathcal{S}G_{\omega}(M)\neq FSHomeo(M)$
\end{corx}



  Theorem \ref{B1} completes the proof of Cojecture \ref{con} on any closed symplectic manifold and gives an affirmative anwser to question the following question posed by Humilière and Seyfaddini:
  \begin{que}\label{que}
  	Are the groups \( Hameo(M, \omega) \) and \( FHomeo(M) \) distinct?
  \end{que}

 This paper is organised as follows: In Section \ref{pre}, we review the construction of norms on $G_{\omega}(M)$ namely,  the Tchuiaga norm $\|\cdot \|^{\infty}$, and  the Hofer-like norm. With the help of the $C^0$ and Hofer-like topologies, we construct a Hofer-like norm on $FSHomeo(M)$. Section \ref{pro} is devoted to prove our results.
 
\section{Preliminaries}\label{pre}
 Let $(M,\omega)$,  be  a closed symplectic surface with genus $g\geqslant1$, and $\mathcal{Z}^1(M)$, the space of closed $1-$forms on $M$. The group $G_{\omega}(M)$ coincides with the time one map of the group of symplectic isotopies $Iso(M,\omega)$. For all  $\alpha\in \mathcal{Z}^1(M)\smallsetminus\{0\}$ and  for any  $\Phi=\{\phi^t\}\in Iso(M,\omega)$, we have 

\begin{equation}\label{ho}
	(\phi^t)^*\alpha -\alpha = d\mathcal{F}^{\alpha}_{\Phi}(t), \qquad \forall t,
\end{equation}
	where $\displaystyle\mathcal{F}^{\alpha}_{\Phi}(t)= \int_{0}^t (\imath_{\dot{\phi^s}}\alpha)\circ \phi^s  ds .$ 	 
Fix $x\in M$, for any $y\in M$, take a curve $\gamma$ from $x$ to $y$. Integrating (\ref{ho}) over $\gamma$ yields
\begin{equation*}
	\int_{\gamma} 	(\phi^t)^*\alpha -\alpha=  \mathcal{F}^{\alpha}_{\Phi}(t)(y)-\mathcal{F}^{\alpha}_{\Phi}(t)(x), \qquad \forall t
\end{equation*}	 
Recently, it has been shown in \cite{tc21} that
	\begin{equation}\label{2.2}
	\Delta(\phi^1,\alpha)_{x}:= \dfrac{1}{\|\alpha\|_{L^2}}\left \langle [\alpha\wedge \omega^{n-1}], \widetilde{S}_{\omega}(\Phi)\right \rangle - \dfrac{Vol(M)}{\|\alpha\|_{L^2}}\mathcal{F}^{\alpha}_{\Phi}(1)(x), 
\end{equation}
 where  $\widetilde{S}_{\omega}	: \widetilde{G_{\omega}(M)}\longrightarrow H^1(\Sigma, \R)$ is the flux homomorphism see \cite{Banyaga78} for more details. From  Proposition $3.8-$\cite{tc21}, if $M$ has boundary,  equation (\ref{2.2})  has an additional term. In \cite{tc21}, Tchuiaga defined a norm on $G_{\omega}(M)$ as follows: 
  $$\|\phi\|^{\infty}=\displaystyle\sup_{\alpha\in \mathcal{B}(1)}\left (\sup_{x\in M}| \widetilde{\Delta}(\phi,\alpha)_{x}|\right ),$$
 for any $\phi\in G_{\omega}(M)$, where $\widetilde{\Delta}(\phi,\alpha)_{x}=\|\alpha\|_{L^2}\Delta(\phi,\alpha)_{x}$ and $\mathcal{B}(1)=\{\alpha\in \mathcal{Z}^1(\Sigma_{g}): \|\alpha\|_{L^2}=1 \}$.

The norm $\|\cdot\|^{\infty}$ extends to the $C^0-$ setting: For any $\phi\in Sympeo(\Sigma_g,\omega)$, 	\begin{eqnarray}\label{norm1}
	\|\phi\|^{\infty} = \sup_{\alpha \in \mathcal{B}(1)} \left( \sup_{x \in M} | \widetilde{\chi}(\phi, \alpha)_{x}| \right),
\end{eqnarray}
where \(\widetilde{\chi}(\phi, \alpha)_{x} = \|\alpha\|_{L^2} \chi(\phi, \alpha)_{x}\) and \(\mathcal{B}(1) = \{\alpha \in \mathcal{Z}^1(\Sigma_{g}) : \|\alpha\|_{L^2} = 1\}\).


  \subsection{Splitting of closed $1-$ forms }
 Let  $B^1(M)$, be the space of exact $1-$forms and $H^1(M,\mathbb{R})$ the  first de Rham cohomology group. Consider  a linear section $\mathcal{S} :   H^1(M,\mathbb{R})\rightarrow \mathcal{Z}^1(M)$  of the natural projection  
 $\mathcal{\pi}: \mathcal{Z}^1(M)\rightarrow  H^1(M,\mathbb{R})$. Any  $\alpha\in \mathcal{Z}^1(M)$ splits as:
 \begin{equation}\label{eq3}
 	\alpha = \mathcal{S}(\mathcal{\pi}(\alpha)) + (\alpha - \mathcal{S}(\mathcal{\pi}(\alpha)))
 \end{equation}
 where $\mathcal{S}(\mathcal{\pi}(\alpha))$ is an $\mathcal{S}$-form and $\alpha-\mathcal{S}(\mathcal{\pi}(\alpha))$ is an  exact form.
 Let $\mathbb H_{\mathcal{S}}(M)$ denote the space of all $\mathcal{S}-$forms. Note that if the linear section $\mathcal{S}$ is trivial then $\mathcal{S}(\mathcal{\pi}(\alpha))$ is a harmonic $1-$form and equation \ref{eq3} is the Hodge decomposition of closed $1-$forms. For more details see \cite{Ba08,Tc18}.
 To  any symplectic isotopy $\Phi=\{\phi^t\}$ corresponds a unique pair $(U,\mathcal{H})$ through the following isomorphisms:
 \begin{eqnarray*}	
 	\begin{array}{cccccc}
 		Iso(M,\omega) &\longrightarrow& Z^1(M)&\xrightarrow{\text{Splitting}}& Z^1(M)&\longrightarrow \mathfrak{T}(M, \omega, \mathcal{S})\\
 		\Phi=\{\phi^t\}_{t}&\longmapsto& \imath_{\dot\phi^t}\omega &\longmapsto& dU^t+\mathcal{H}^t &\longmapsto  (U,\mathcal{H}),
 \end{array}\end{eqnarray*}
 where $U=\{U^t\}_{t}\in C^\infty_{0}(M\times[0,1],\R)$ ( the set of zero mean smooth functions on $M\times[0,1]$), $\mathcal{H}=\{\mathcal{H}^t\}_{t}$ is a  smooth path in $\mathbb H_{\mathcal{S}}(M)$  and  $\mathfrak{T}(M, \omega, \mathcal{S})$ is the set of all generators $(U,\mathcal{H})$ of symplectic isotopies.
 In the sequel, we write $\phi_{(U,\mathcal{H})}$ to mean the isotopy $\Phi$ is generated by $(U,\mathcal{H})$.
 The set $\mathfrak{T}(M, \omega, \mathcal{S})$ forms a  group w.r.t the following rule: 
 
 \begin{equation}\label{Productrule}
 	(U,\mathcal{H})\Join_{\mathcal{S}}(V ,\mathcal{K}) = 	\left ( U + V\circ\phi_{(U,\mathcal{H})}^{-1} 
 	+ \widetilde{\mathcal{F}}^{\mathcal{K}}_{\phi_{(U,\mathcal{H})}^{-1}}(t),\quad \mathcal{H} 
 	+ \mathcal{K} \right ), 
 \end{equation}
 the inverse of $(U,\mathcal{H})$, denoted $\overline{(U,\mathcal{H})}$, is given by 
 \begin{equation}\label{0Productrule}
 	\overline{(U,\mathcal{H})} = \left (- U\circ\phi_{(U,\mathcal{H})} - 
 	\widetilde{\mathcal{F}}^{\mathcal{H}}_{\phi_{(U,\mathcal{H})}}(t), \quad	-\mathcal{H}\right )
 \end{equation}
 where  $ \displaystyle \widetilde{\mathcal{F}}^{\mathcal{H}}_{\phi_{(U,\mathcal{H})}}(t)=\int_0^t \mathcal H^t( \dot\psi^s_{(\mathcal U, \mathcal H)})\circ \psi^s_{(\mathcal U, \mathcal H)} ds -\frac{\displaystyle \int_M\left(\int_0^t \mathcal H^t( \dot\psi^s_{(\mathcal U, \mathcal H)})\circ \psi^s_{(\mathcal U, \mathcal H)} ds\right) \omega}{\displaystyle\int_M\omega}$
 for all $t\in [0,1].$ 
 Given a symplectic isotopy $\Phi=\{\phi^t\}$ generated by $ (U,\mathcal{H}),$ 
 the  $L^{(1,\infty)}-$version of the Hofer-like length of  $\Phi$ is given  by
 \begin{equation}\label{blg2}
 	l^{(1,\infty)}_{\kappa, \mathcal{S}}(\Phi) 
 	= \displaystyle \int_{0}^1\left ( \nu^B(dU^t) + \kappa\| \mathcal{H}^t\| _{L^2}\right )dt,
 \end{equation}
 where $\nu^B$ is any norm on $B^1(M)$, the coefficient  $\kappa$ is a positive real number. Here, we choose $ \nu^B$ to be  the oscillation  norm for more details see \cite{Tc18}.
 The above Hofer-like length gives rise to the Hofer-like energy $e_{HL}$,   which assigns to any $\psi \in G_{\omega}(M)$:
 \begin{eqnarray*}
 	e_{HL}(\psi)=\inf l^{(1,\infty)}_{\kappa, \mathcal{S}}(\Phi) 	
 \end{eqnarray*}
 where the infimum is taken over all symplectic isotopies whose time one map is $\psi$. The $L^{(1,\infty)}-$ Hofer-like norm of $\psi$ is given by 
 \begin{eqnarray*}
 	\|\psi\|_{HL}= \frac{e_{HL}(\psi)+e_{HL}(\psi^{-1})}{2},
 \end{eqnarray*}
 the latter norm induces a right invariant distance $d_{HL}$ on $G_{\omega}(M)$ given  by $$d_{HL}(\phi,\psi)=\|\phi\circ\psi^{-1}\|_{HL} \quad \forall \phi,\psi\in G_{\omega}(M). $$ 
 The length $l^{(1,\infty)}_{\kappa, \mathcal{S}}$ also induces a metric $D^{1}_{\kappa, \mathcal{S}}$ on the $Iso(M,\omega)$ given by:
 $$D^{1}_{\kappa, \mathcal{S}}((U,\mathcal{H}), (V ,\mathcal{K})) := \frac{D_0^{\kappa, \mathcal{S}}((U,\mathcal{H}), (V ,\mathcal{K})) + D_0^{\kappa, \mathcal{S}}(\overline{(U,\mathcal{H})}, \overline{(V ,\mathcal{K})})}{2}, \quad \forall\kappa>0,$$
 where $$D_0^{\kappa, \mathcal{S}}((U,\mathcal{H}), (V ,\mathcal{K})):= \displaystyle \int_{0}^1\left ( osc(U_t- V_t) + \kappa\| \mathcal{H}_t - \mathcal{K}_t\| _{L^2}\right )dt,$$ 
 
 See \cite{Tc18} for more details on the Hofer-like metrics. In this paper, we will assume $\kappa = 1$. 

	\subsection{From smooth to $C^0$}
$C^0-$symplectic topology began with Eliashberg-Gromov rigidity theorem \cite{Elia,gro} which states that  $Symp(M,\omega)$ is $C^0-$closed in $Diff(M)$ the group of diffeomorphisms of $M$. This inspired Oh-M\"uller \cite{Oh-M07} to define  $Sympeo(M,\omega)$, the symplectic homeoeomorphism group as  $$\overline{Symp(M,\omega)}^{C^0}\subseteq Homeo(M),$$
where  $Homeo(M)$ is the group of homeomorphisms of $M$. The $C^0-$topology  on $Homeo(M)$ is the topology induced by the  distance:
$$d_0(f,h) = \max(d_{C^0}(f,h),d_{C^0}(f^{-1},h^{-1})), \quad \forall f,h\in Homeo(M)$$ 
with  $\displaystyle d_{C^0}(f,h) =\sup_{x\in M}d_g (h(x),f(x))$ ).\\ 
On the space of all continuous paths $\lambda:[0,1]\rightarrow Homeo(M)$ such that $\lambda(0) = id_M$,  the $C^0-$topology is  the  topology induced by the distance: 
\begin{eqnarray*}
	\bar{d}(\lambda,\mu) = \max_{t\in [0,1]}d_0(\lambda(t),\mu(t)).
\end{eqnarray*}
The distance $\bar{d}$ is neither left nor right invariant however, it satisfies the following: 

\begin{proposition}\label{pro01} \cite{TM} $ $ 
	Let $\Theta$ be a smooth  isotopy of diffeomorphisms of $M$. Then there exists a constant $C^\Theta>0$ depending on $\Theta$  such that 
	for all paths  $\Phi$, $\Psi$ in $Homeo(M)$, we have 
	$$\bar{d}(\Phi\circ\Theta, \Psi\circ\Theta)\leqslant C^\Theta \bar{d} (\Phi, \Psi).$$
\end{proposition}

Given $\phi\in FSHomeo(M)$, define 
\begin{eqnarray}
	\|\phi\|_{\widetilde{HL}}:= \liminf\limits_{i\longrightarrow\infty} \|\phi_{i}\|_{HL},
\end{eqnarray}
where the infimum is taken over all $(\phi_{i})_{i}\subset G_{\omega}(M)$ such that $\phi_{i}\xrightarrow{C^0}\phi$.
The rule $\|\cdot\|_{\widetilde{HL}}$  defines a  Hofer-like norm on  $FSHomeo(M)$ indeed:
\begin{enumerate}
	\item Positivity: clearly   $\|\phi\|_{\widetilde{HL}}\geqslant0$, for any $\phi\in FSHomeo(M)$
	\item Non-degeneracy: Let $\phi\in FSHomeo(M)$ and $(\phi_{i})_{i}\subset G_{\omega}(M)$ such that $\phi_{i}\xrightarrow{C^0}\phi$.  Assume 
	\begin{eqnarray*}
		0=\|\phi\|_{\widetilde{HL}}= \liminf\limits_{i\longrightarrow\infty} \|\phi_{i}\|_{HL}.
	\end{eqnarray*}
	By definition of $\liminf\limits$, we have for any $\epsilon>0$, there exists $\delta>0$, such that  for all $(\phi_{i})_{i}\subseteq G_{\omega}(M)$, satisfying $d_{0}(\phi_{i},\phi)<\delta$, we have  $\|\phi_{i}\|_{HL}-\epsilon < \|\phi\|_{\widetilde{HL}}$. By assumption, $\|\phi\|_{\widetilde{HL}}=0$, hence $\|\phi_{i}\|_{HL}-\epsilon < 0$, i.e, $0\leqslant\|\phi_{i}\|_{HL}<\epsilon$. Letting  $\epsilon$  tend to zero,  we  get that $\|\phi_{i}\|_{HL}\longrightarrow 0$, as $i\longrightarrow \infty$. Therefore, $\phi_{i}\xrightarrow{C^0}\phi$, and $\|\phi_{i}\|_{HL}\longrightarrow 0$, as $i\longrightarrow \infty$ hence by Theorem $3.3-$\cite{Tc18},we have that $\phi=id$.
	\item Triangle inequality:  let $\phi$, $\psi \in FSHomeo(M)$. By definition, $\phi=\lim\limits_{C^0}\phi_i$,  and $\psi=\lim\limits_{C^0}\psi_i$, where $(\phi_{i})_{i}$ and $(\psi_{i})_{i}$ are sequences in $G_{\omega}(M)$ whose Hofer-like norms are bounded. By triangle inequality of $\|\cdot\|_{HL}$, one has
	\begin{eqnarray}\label{1}
		\|\phi_{i}\circ\psi_{i}\|_{HL}\leqslant \|\phi_{i}\|_{HL}+ \|\psi_{i}\|_{HL}.
	\end{eqnarray}
	Taking the infimum  over all  $\phi_{i}\circ\psi_{i}$ which $C^0$ converges to  $\phi\circ \psi$, on both sides of (\ref{1}) we have 
	\begin{eqnarray*}
		\liminf\limits_{i\longrightarrow\infty} \|\phi_{i}\circ\psi_{i}\|_{HL}\leqslant\liminf\limits_{i\longrightarrow\infty}(\|\phi_{i}\|_{HL}+ \|\psi_{i}\|_{HL})\leqslant \liminf\limits_{i\longrightarrow\infty}\|\phi_{i}\|_{HL}+\liminf\limits_{i\longrightarrow\infty}\|\phi_{i}\|_{HL}
	\end{eqnarray*}
\end{enumerate}
The norm $\|\cdot\|_{\widetilde{HL}}$ induces a right invariant distance on $FSHomeo(M)$ defined as:
\begin{eqnarray*}
	\tilde{d}_{HL}(\psi, \phi)=\liminf\limits_{i\longrightarrow\infty}\|\phi\circ\psi^{-1}\|_{HL}
\end{eqnarray*}



\section{Proof of Results}\label{pro}

\subsection{Proof of Theorem \ref{A1}}

The  proof  Theorem \ref{A1} consists of  constructing a homeomorphism $\phi\in  \overline{G_{\omega}(\Sigma_{g})}$, 
such that  $\|\phi\|^\infty=\infty$.  Then any such  $\phi$ can not be in $FSHomeo(\Sigma_{g})$.

	\begin{lemma}\label{le}
	For any $\phi\in FSHomeo(\Sigma_{g})$, we have that $\|\phi\|^{\infty}$ is bounded.
\end{lemma}
\begin{proof}
	The proof of Lemma \ref{le} follows from the proof of Theorem $6.5-$\cite{tc21}
\end{proof}
\begin{lemma}\label{lem3.1}
	There exists a sequence $(\phi_{i})_{i}\subset G_{\omega}(\Sigma_{g})$, such that $\phi_{i}\xrightarrow{C^0}\phi$ and $\|\phi_{i}\|^\infty\longrightarrow\infty$, as $i\longrightarrow\infty$
	
	
\end{lemma}
\begin{proof}
	We proceed by induction on $g$
	\begin{itemize}
		\item [Step 1:]
		For $g=1$,  consider the $2$ torus $\mathbb{T}^2=\mathbb{S}^1\times\mathbb{S}^1 $ with coordinate system $(\theta_{1},\theta_{2})$ and  consider the standard  symplectic form $\omega=d\theta_{1}\wedge d\theta_{2}$ where  $\theta_{1}$, $\theta_{2}\in [0, 2\pi] $.  
		 Let $h:(0,2\pi]\to \mathbb{R}$ be a function of $\theta_2$, satisfying the following properties
			\begin{enumerate}
			\item \( h(\theta_2) > 0 \) for all \( \theta_2 \in (0, 2\pi] \),
			\item \( \lim_{\theta_2 \to 0^+} h(\theta_2) = 0 \),
			\item \( \int_0^{2\pi} h(\theta_2) \, d\theta_2 = \infty \).
		\end{enumerate}

		 Consider the isotopy $\{\phi^t\}_t$  on  $\mathbb{T}^2\backslash\{(\theta_{1},0)\}$ given by 
		$$\phi^t(\theta_{1},\theta_{2}):= (\theta_{1}+th(\theta_{2}), \theta_{2}), \qquad \forall t.$$
		Equip  $\mathbb{T}^2$ with the flat riemannian metric $g_{0}$, then the $1-$forms $d\theta_{1}$ and $d\theta_{2}$ are harmonic. Let $f$ be the antiderivative of $h$.
		The vector field generating the isotopy \( \{\phi^t\}_t \) satisfies:
		\[
		\iota_{\dot{\phi}^t} \omega = h d\theta_2.
		\]
		Using the Hodge decomposition, we can write:
		\[
		hd\theta_2 = d(h\theta_2 - \theta_2) + d\theta_2.
		\]
		Thus, the isotopy \( \{\phi^t\}_t \) is generated by the pair:
		$
		(U, \mathcal{H}) = (h\theta_2 - \theta_2, d\theta_2),
		$
		and its flux satisfies:
		$
		\widetilde{S}_\omega(\{\phi^t\}_t) = [d\theta_2] \neq 0.
		$
		The isotopy \( \{\phi^t\}_t \) can be extended to a non-Hamiltonian symplectic flow \( \{\psi^t\}_t \) on all of \( \mathbb{T}^2 \) by fixing the points \((\theta_1, 0)\) i.e, 
		\[
		\psi^t(\theta_1, \theta_2) =
		\begin{cases}
			\phi^t(\theta_1, \theta_2), & \text{if } \theta_2 > 0, \\
			(\theta_1, 0), & \text{if } \theta_2 = 0.
		\end{cases}
		\]
		
		The time-1 map \(\psi^1\) belongs to the group of symplectic homeomorphisms, \(\text{Sympeo}_0(\mathbb{T}^2, \omega)\). To verify this:
		\begin{itemize}
			\item \textbf{Regularity:} $\psi^1$ is differentiable everywhere except at the points $(\theta_{1},0)$  
			\item \textbf{Smooth approximation.} Consider the  sequence $(\phi_i)_i\subseteq G_{\omega}(\mathbb{T}^2)$ defined as follows: for all $i\in \mathbb{N}^*$, let \( h_i: [0, 2\pi] \to \mathbb{R} \) be a smooth function such that:
			\[
			h_i(\theta_2) =
			\begin{cases}
				h(\theta_2), & \text{if } \theta_2 \in \left[\frac{1}{i}, 2\pi\right], \\
				\text{smoothly vanishing}, & \text{if } \theta_2 \in [0, \frac{1}{i}],
			\end{cases}
			\]
			with \( h_i \leq h_{i+1} \). Then for each $i$ set $\phi_{i}(\theta_{1},\theta_{2}) := (\theta_{1}+ h_{i}(\theta_{2}), \theta_{2})$ and derive that    $\phi_{i}\circ(\psi^1)^{-1}\xrightarrow{C^0}id$, as $i\longrightarrow\infty$.  
		\end{itemize}    		
Now, let  $\alpha=d\theta_{1}$, then  $(\phi_{i})^*\alpha -\alpha =dh_{i}$. Pick $x\in\mathbb{T}^2$  (fixed),  and take  a $C^1-$curve  $\gamma$ from $x$  to any point  $y\in \mathbb{T}^2$, we have the following computation:
		
		\begin{eqnarray}
			\int_{\mathbb{T}^2}\left (\int_{\gamma} 	(\phi_{i})^*\alpha -\alpha\right )\omega &=&	\int_{\mathbb{T}^2}\left (\int_{\gamma} dh_{i} \right )\omega\nonumber\\
			&=&\int_{\mathbb{T}^2} \left (h_{i}(y)-h_{i}(x)\right )\omega\label{3.1}
		\end{eqnarray}
		Set $A$, to be the area of $\mathbb{T}^2$.
		Since $x$ is fixed, (\ref{3.1}) becomes 
		\begin{eqnarray}\label{3.2}
			\displaystyle \int_{\mathbb{T}^2}\left (\int_{\gamma} 	(\phi_{i})^*\alpha -\alpha\right )\omega &=& \int_{\mathbb{T}^2} \left (h_{i}(y)\omega- A h_{i}(x)\right )\nonumber\\
			&=&\int_{[0,2\pi]\times[0,2\pi]}h_{i}(y)\omega- A h_{i}(x).
		\end{eqnarray} 
		Taking the limit as $i$ tend to $\infty$ in equation  (\ref{3.2}), we have 
		$\Delta(\phi_{i}, \alpha ) \longrightarrow \infty$ as $i\longrightarrow\infty$ hence $\|\phi_{i}\|^\infty\longrightarrow \infty$ as $i\longrightarrow\infty$, i.e, $\|\psi^1\|^\infty=\infty$.
		\item [Step 2: ] For $g=2$, consider $2$, tori $\mathbb{T}^2_{1}$, $\mathbb{T}^2_{2}$ and flows $\{\phi^t_{1}\}_t$ and $\{\phi^t_{2}\}_{t}$ defined on each torus as in Step $1$. Remove discs  $\mathbb{D}_{1}$ and $\mathbb{D}_{2}$ of radius $R<1$ from  $\mathbb{T}^2_{1}$ and $\mathbb{T}^2_{2}$ respectively.  A $2-$ genus surface $\Sigma_{2}$ is obtained  by gluing  $\mathbb{T}^2_{1}\backslash \mathbb{D}_{1}$, and $\mathbb{T}^2_{2}\backslash \mathbb{D}_{2}$ along the boundaries of $\mathbb{D}_{1}$ and $\mathbb{D}_{2}$.  Define $\varphi\in Sympeo_0(\Sigma_2,\omega)$  as follows:
		\begin{eqnarray}
			\varphi=\left \{\begin{array}{cccccc}
				\phi_{1}^1  & \text{on }& \mathbb{T}^2_{1}\smallsetminus \mathbb{D}_{1}\\
				\phi^1_{2} & \text{on}& \mathbb{T}^2_{2}\smallsetminus \mathbb{D}_{2}
			\end{array}\right .,
		\end{eqnarray}
		such that $\phi^1_{j}= id$ on $\partial\mathbb{D}_{j}$, $j=1,2$.
		Consider a smooth symplectic map $\varphi_{i}$ defined as follows
		\begin{eqnarray}
			\varphi_{i}=\left \{\begin{array}{cccccc}
				\phi_{1,i}^{1}  & \text{on }& \mathbb{T}^2_{1}\smallsetminus \mathbb{D}_{1}\\
				\phi^1_{2,i} & \text{on}& \mathbb{T}^2_{2}\smallsetminus  \mathbb{D}_{2}
			\end{array}\right .,
		\end{eqnarray}
		where $\phi_{j,i}^1$, $j=1,2$ coincide with the map  smooth $\phi_{i}$ constructed in Step $1$ and $\phi_{j,i}^1=id$ on $\partial\mathbb{D}_{j}$. Then $\varphi_{i}\xrightarrow{C^0}\varphi$, as $i\longrightarrow\infty$.
		
		For each $j\in\{1,2\}$ set $\alpha_{j}= d\theta^{j}_{1}$.
		Consider the projection,  $Pr_{j}:\Sigma_{2}\longrightarrow\mathbb{T}^2_{j}\backslash \mathbb{D}_{j}$.
		Let $\alpha\in Z^1(\Sigma_{2})$ be given by 
		$\alpha = Pr_1^*(\alpha_1)+ Pr_2^*(\alpha_2)$.  
		Let  $\gamma$ a $C^1-$curve from $x$ (fixed) to $y$, in $\Sigma_{2}$, then 
		\begin{eqnarray}\label{3.5}
			\widetilde{\Delta}(\varphi_{i},\alpha)=\left \{\begin{array}{cccccc}
				\widetilde{\Delta}(	\phi_{1,i}^1,\alpha)  & \text{on }& \mathbb{T}^2_{1}\smallsetminus \mathbb{D}_{1}\\
				\widetilde{\Delta}(	\phi_{2,i}^1,\alpha) & \text{on}& \mathbb{T}^2_{2}\smallsetminus  \mathbb{D}_{2}
			\end{array}\right .,
		\end{eqnarray} 
		where
		\begin{eqnarray}\label{3.6}
			{\hspace{-3cm}\widetilde{\Delta}(	\phi_{j,i}^1,\alpha)= \langle \widetilde{S}_{\omega}(\phi_{j,i}^t), [\alpha]\rangle + \int_{\partial\mathbb{D}_{j}}\left (\int_{0}^1 \mathcal{F}_{\phi_{j,i}^t}^{\alpha}(t)  (\phi_{j,i}^t)^*(\imath_{\dot\phi^t_{j,i}}\omega)  dt\right )  - Area(\mathbb{T}^2_{j}\smallsetminus \mathbb{D}_{j})  \mathcal{F}_{\phi_{j,i}^t}^{\alpha}(1)(x).}
		\end{eqnarray} 
		See Proposition $3.8-$\cite{Tch21} for details.  Since for each $j$, $\phi_{j,i}= id$ on $\partial\mathbb{D}_{j}$, then 
		$$\displaystyle \int_{\partial\mathbb{D}_{j}}\left (\int_{0}^1 \mathcal{F}_{\phi_{j,i}^t}^{\alpha}(t)  (\phi_{j,i}^t)^*(\imath_{\dot\phi^t_{j,i}}\omega)  dt\right ) =0$$
		
		Using (\ref{3.1}) and (\ref{3.2}), one can show that the right hand side of equation (\ref{3.6})  goes to $\infty$, as $i\longrightarrow\infty$. 
		Therefore,
		taking the limit as $i$ tends to $\infty$ in  (\ref{3.5}) yields  $\widetilde{\Delta}(\varphi_{i},\alpha) \longrightarrow \infty$ as $i\longrightarrow\infty$ and  $\|\varphi\|^\infty=\infty$
		
		\item [Step 3] Assume that on a surface  $\Sigma_{g-1}$ of genus $g-1$, there exists a sequence $(\psi_{i})_{i}\subset G_{\omega}(\Sigma_{g-1})$, such that $\psi_{i}\xrightarrow{C^0}\psi$ and $\|\psi_{i}\|^\infty\longrightarrow\infty$, as $i\longrightarrow\infty$.   
		Remove a disc $\mathbb{D}_{g-1}$ of radius $R<1$ from $\Sigma_{g-1}$ and  consider the punctured torus  $\mathbb{T}^2\smallsetminus \mathbb{D}$ with $\mathbb{D}$ a disc of radius $R<1$. A $g-$genus surface is formed by gluing  $\Sigma_{g-1}\smallsetminus \mathbb{D}_{g-1}$  and $\mathbb{T}^2\smallsetminus\mathbb{D}$ along the boundaries of $\mathbb{D}_{g-1}$ 	and $\mathbb{D}$.  Define $\vartheta\in Sympeo_0(\Sigma_{g},\omega))$  as follows:
		\begin{eqnarray}
			\vartheta=\left \{\begin{array}{cccccc}
				\phi & \text{on }& \mathbb{T}^2\smallsetminus \mathbb{D}_{1}\\
				\psi & \text{on}& \Sigma_{g-1}\smallsetminus\mathbb{D}_{g-1}
			\end{array}\right .,
		\end{eqnarray}
		where  $\phi$ is constructed as in Step $2$
		and $\psi=id$ on $\partial \Sigma_{g-1}\backslash \mathbb{D}_{g-1}$. Following similar arguments from  Step $2$, one constructs a sequence $(\vartheta_{i})_{i}$ of symplectic diffeomorphisms  and a closed $1-$ form $\alpha$ such that 
		$\vartheta_{i}\xrightarrow{C^0}\vartheta$ and $\widetilde{\Delta}(\vartheta_{i},\alpha)$ tends to infinity as $i\longrightarrow \infty$ that is, $\|\vartheta_{i}\|^\infty\longrightarrow\infty$, as $i\longrightarrow\infty$. 
		Therefore, by mathematical induction, on any closed  symplectic surface $\Sigma_{g}$, of genus $g\geqslant1$,  there is a sequence $(\phi_{i})_{i}\subset G_{\omega}(\Sigma_{g})$, such that $\phi_{i}\xrightarrow{C^0}\phi$ and $\|\phi_{i}\|^\infty\longrightarrow\infty$, as $i\longrightarrow\infty$.\end{itemize}\end{proof}
	
	
In order to complete the proof  Theorem \ref{A1}, we   now show that $FSHomeo(M)$  is a normal subgroup  of $Sympeo_0(M)$, for all closed symplectic manifolds $(M,\omega)$. 
	
	\begin{lemma}
		For any closed symplectic surface $(M,\omega)$, $FSHomeo(M)$  is a normal subgroup  of $Sympeo_0(M)$.
	\end{lemma}	
	\begin{proof}
		Let $\varphi\in FSHomeo(M)$ and  $\psi \in Sympeo_0(M)$. By definition, $\psi=\lim\limits_{C^0} \psi_{i}$ with  $(\psi_{i})_i\subseteq G_{\omega}(M)$.  Consider a   smooth sequence  $\left (\psi_{(V_{i}, \mathcal{K}_{i})}\right )_{i}\subseteq Iso(M,\omega)$ whose  sequence of time one maps $(\psi_{i})_{i}$ converge to $\varphi$ in the $C^0-$topology   and  the sequence  $(V_{i},  \mathcal{K}_{i})_{i}$ is bounded w.r.t $D^{1}_{\kappa, \mathcal{S}}$ .  The sequence $(V_{i}\circ\psi_i+ \widetilde{\mathcal{F}}^{\mathcal{K}_{i}}_{\psi_{(V_{i}, \mathcal{K}_{i})}}(t)),\quad \mathcal{K}_{i})\subset\mathfrak{T}(\Sigma_{g}, \omega, \mathcal{S})$ generates the sequence  of isotopies $(\psi^{-1}_{i}\circ\psi_{(V_{i},  \mathcal{K}_{i})}\circ\psi_{i})_{i}$ whose time one maps  $C^0-$converges to $\psi^{-1}\circ\varphi\circ\psi$ and 
		\begin{eqnarray}
			\int_{0}^1\left ( osc\left [V_{i}^t\circ\psi_i+ \widetilde{\mathcal{F}}^{\mathcal{K}_{i}}_{\psi_{(V_{i}, \mathcal{K}_{i})}}(t)\right ]+ \|\mathcal{K}_{i}^t\|_{L^2}\right )dt\leqslant 		\int_{0}^1osc(V_{i}^t\circ\psi_i)dt+ 	\int_{0}^1osc\left (\widetilde{\mathcal{F}}^{\mathcal{K}_{i}}_{\psi_{(V_{i}, \mathcal{K}_{i})}}(t)\right )dt +	\int_{0}^1\|\mathcal{K}_{i}^t\|_{L^2}dt\nonumber
		\end{eqnarray}
		Combining equations $(2.10)$ and $(2.25)$ in \cite{Tc18}, we get that 
		
		\begin{eqnarray}
			\int_{0}^1\left ( osc\left [V_{i}^t\circ\psi_i+ \widetilde{\mathcal{F}}^{\mathcal{K}_{i}}_{\Psi_i}(t)\right ]+ \|\mathcal{K}_{i}^t\|_{L^2}\right )dt\leqslant 	\int_{0}^1\left (osc(V_{i}^t)dt+ B K(g_{0}) \|\mathcal{K}_{i}^t\|_{L^2}\right )dt+ \int_{0}^1\|\mathcal{K}_{i}^t\|_{L^2}dt\nonumber
		\end{eqnarray}
		for some positive constants $B$ and $K(g_{0})$ which depend on a riemannian metric $g_{0}$ on $M$ . Therefore, 
		\begin{eqnarray}\label{eq6}
			\hspace{-1cm}	\int_{0}^1\left ( osc\left [V_{i}^t\circ\psi_i+ \widetilde{\mathcal{F}}^{\mathcal{K}_{i}}_{\psi_{(V_{i}, \mathcal{K}_{i})}}(t)\right ]+ \|\mathcal{K}_{i}^t\|_{L^2}\right )dt&\leqslant& 	\int_{0}^1osc(V_{i}^t)dt+(B K(g_{0})+1)\int_{0}^1 \|\mathcal{K}_{i}^t\|_{L^2} dt\nonumber\\
			&\leqslant& 2(B K(g_{0})+1)\int_{0}^1( osc(V_{i}^t)+\|\mathcal{K}_{i}^t\|_{L^2}) dt.
		\end{eqnarray}
		The right hand side of inequality (\ref{eq6}) is bounded since $(V_{i},\mathcal{K}_{i})$ is bounded w.r.t $D^{1}_{\kappa, \mathcal{S}}$.
	\end{proof}

\subsection{Proof of Theorem \ref{B1}}
We need the following Lemma in order to prove Theorem\ref{bg}
\begin{lemma}\label{lbg}
	Let $\Phi=\{\phi^t\}$ be a symplectic isotopy . For every $E>0$,  such that $l^{(1,\infty)}_{\kappa, \mathcal{S}}(\Phi)<E$, there exist a symplectic isotopy $\Phi'=\{(\phi')^t\}$ such  that 

	\begin{enumerate}
		\item $\bar{d}(\Phi, \Phi')$ can be made arbitrarily small
		\item $l^{(1,\infty)}_{\kappa, \mathcal{S}}(\Phi')<E+\epsilon$,  for some arbitrarily positive $\epsilon$
		\item for any $\alpha\in \mathcal{Z}^1(M))\smallsetminus\{0\}$, and any symplectic isotopy $\Psi=\{\psi^t\}$ there exists a constant $C>0$ such  that\  $\Delta((\phi')^1\circ(\psi^1)^{-1},\alpha)>C E-\epsilon $. 
	\end{enumerate}
\end{lemma}
\begin{proof}
	

Assume $Vol(M)=1$, let $\alpha\in\cal{Z}^1(M)\smallsetminus\{0\}$ and $\Phi=\{\phi^{t}\}$ a symplectic isotopy generated by $(U,\cal{H})$ and  $l^{(1,\infty)}_{\kappa, \mathcal{S}}(\Phi) < E$. By Theorem $6.5-$\cite{tc21}, we have $\|\phi\|^{\infty}\leqslant C l^{(1,\infty)}_{\kappa, \mathcal{S}}(\Phi)$ with  $C=3\max\{C(\omega), \eta, 2\}$.  The Hofer norm of the function  $\displaystyle\mathcal{F}^{\alpha}_{\Phi}(1)$ is given by
\begin{eqnarray}\label{cal}
	osc\left (\mathcal{F}^{\alpha}_{\Phi}(1)\right )&\leqslant& 2\sup_{x\in \Sigma_g}\left |\int_{0}^1(\imath_{\dot{\phi}^t}\alpha)(x) dt\right | \nonumber\\
	&\leqslant&2\sup_{x}\left |\int_{0}^1(\imath_{\dot{\phi}^t}\omega)(X_{\alpha})(x) dt\right |	
\end{eqnarray}
where $X_{\alpha}$ is the symplectic  vector field defined by $\imath_{X_{\alpha}}\omega= \alpha$. From the splitting in (\ref{eq3}), inequality (\ref{cal})  becomes
\begin{eqnarray}\label{cal1}
	osc\left (\mathcal{F}^{\alpha}_{\Phi}(1)\right )&\leqslant& 2\sup_{x}\left |\int_{0}^1(dU^s+\mathcal{H}^t)(X_{\alpha})(x) dt\right |\nonumber\\
	&\leqslant& 2 \left (\sup_{x}\left |\int_{0}^1(dU^t)(X_{\alpha})(x)dt\right |	+\sup_{x}\left |\int_{0}^1(\mathcal{H}^t)(X_{\alpha})(x)dt\right |	\right )
\end{eqnarray}
From  the lines of proof of Proposition $6.1-$\cite{tc21}, we have 
\begin{eqnarray}\label{3.21}
	\sup_{x}\left |\int_{0}^1(dU^t)(X_{\alpha})(x)dt\right |\leqslant 2\max_{t}(osc(U^t)	
\end{eqnarray} 
By Proposition $6.3-$\cite{tc21}, for all $\alpha\in \mathcal{Z}^1(M)$, and a harmonic vector field $X$, one can find a positive constant $C(\alpha)$ such that 
\begin{eqnarray}
	\sup_{x \in M}|\alpha(X)(x)|\leqslant C(\alpha)\|\imath_{X}\omega\|_{L^2}.
\end{eqnarray}
Therefore, 	
\begin{eqnarray}\label{3.22}
	\sup_{x}\left |\int_{0}^1(\mathcal{H}^t)(X_{\alpha})(x)dt\right |&=&\sup_{x}\left |\int_{0}^1\imath_{X_{\mathcal{H}^t}}\omega)(X_{\alpha})(x)dt\right |\nonumber\\
	&=&\sup_{x}\left |\int_{0}^1(\imath_{X_{\alpha}}\omega)(X_{\mathcal{H}^t})(x)dt\right |\nonumber\\
	&\leqslant& \sup_{x,t}|(\imath_{X_{\alpha}}\omega)(X_{\mathcal{H}^t})(x)|\nonumber\\
	&=& C(\alpha)\sup_{t}\|\mathcal{H}^t\|_{L^2}
\end{eqnarray}
Putting (\ref{3.21}) and (\ref{3.22}) in  (\ref{cal1}), we  have that
\begin{eqnarray*}
	osc\left (\mathcal{F}^{\alpha}_{\Phi}(1)\right )\leqslant 2\max\{2,C(\alpha)\}l^{(1,\infty)}_{\mathcal{S}}\left (\Phi\right )< 2\max\{2,\eta\}E\leqslant CE,
\end{eqnarray*}
 where $\eta$ is as in  Lemma $6.4$ item $2$ .
 Let $\epsilon>0$, be arbitrary and $D_{\epsilon}$ a disc   of radius $\epsilon$ in $M$, set $\displaystyle a(t)=\sup_{x\in D_{\epsilon}}\mathcal{F}^{\alpha}_{\Phi}(t)(x)$  and  $ b(t)=\displaystyle\inf_{x\in D_{\epsilon}}\mathcal{F}^{\alpha}_{\Phi}(t)(x)$ for each $t\in[0,1]$. Consider a function $h:M\to\mathbb{R}$ compactly supported in $D_{\epsilon}$ such that $h(x)\leqslant 1$ for all $x\in D_{\epsilon}$ and $osc(h)\leqslant\frac{\epsilon}{CE+\epsilon}$.
Let $\phi_{(K_{1},0)}$ and $\phi_{(K_{2},0)}$ be hamiltonian isotopies  
generated by the functions  $K_1(x,t)= \left (a(t)+\frac{\epsilon-3CE}{2}\right )h(x)$  and $K_2(x,t)= \left ( b(t)-\frac{\epsilon-3CE}{2}\right )h(x)$ respectively. The symplectic isotopies  $\{\phi_{1}^t\}=\Phi_1=:\phi_{(K_{1},0)}\circ\Phi$ and $\{\phi_{2}^t\}=\Phi_{2}=:\phi_{(K_{2},0)}\circ\Phi$  are good candidates for $\Phi'$. Indeed, for any symplectic isotopy $\Psi=\{\psi^t\}_{t}$ and  $\alpha\in \mathcal{Z}^1(M)$, we have
\begin{eqnarray}\label{qs0}
	\Delta(\phi_{1}^1\circ(\psi^1)^{-1},\alpha)_{y}=\Delta(\phi_{2}^1\circ(\psi^1)^{-1},\alpha)_{y}+ \Delta(\phi_{1}^1\circ(\phi_{2}^1)^{-1},\alpha)_{\phi_{2}^1\circ(\psi^1)^{-1}(y)}, 
\end{eqnarray}
for some $y\in M$. But 
\begin{eqnarray}\label{ho11}
	\Delta(\phi_{1}^1\circ(\phi_{2}^1)^{-1},\alpha)_{\phi_{2}^1\circ(\psi^1)^{-1}(y)}&=&\dfrac{1}{\|\alpha\|_{L^2}}\left \langle [\alpha], \widetilde{S}_{\omega}(\Phi_{1}\circ\Phi_{2}^{-1})\right \rangle - \dfrac{Vol(M)}{\|\alpha\|_{L^2}}\mathcal{F}^{\alpha}_{\Phi_{1}\circ\Phi_{2}^{-1}}(1)(\phi_{2}^1\circ(\psi^1)^{-1}(y))\nonumber\\
	&=& - \dfrac{1}{\|\alpha\|_{L^2}}\mathcal{F}^{\alpha}_{\phi_{K_1}\circ\phi_{K_2}^{-1}}(1)(\phi_{2}^1\circ(\psi^1)^{-1}(y))\nonumber\\
	&=&-\dfrac{1}{\|\alpha\|_{L^2}}\int_{0}^1 \alpha(\dot{\overbrace{\phi^t_{K_{1}}\circ(\phi^t_{K_{2}})^{-1}}})\circ \phi^t_{K_{1}}\circ(\phi^t_{K_{2}})^{-1}(y)dt\nonumber\\
	&=&\dfrac{1}{\|\alpha\|_{L^2}}\int_{0}^1 \left (\imath_{\dot{\overbrace{\phi^t_{K_{1}}\circ(\phi^t_{K_{2}})^{-1}}}}\omega\right )(X_{\alpha})\circ \phi^t_{K_{1}}\circ(\phi^t_{K_{2}})^{-1}(y)dt\nonumber\\
	&=&\dfrac{1}{\|\alpha\|_{L^2}}\int_{0}^1 d(K_{1}\sharp\bar{K}_{2})^{t}(X_{\alpha}(z))dt\nonumber\\
	&=&\dfrac{1}{\|\alpha\|_{L^2}}\int_{0}^1 \frac{d}{ds}((K_{1}\sharp\overline{K}_{2})^{t}(\phi^s_{\alpha}(z))dt
\end{eqnarray}
Normalizing $\alpha$ such that $\|\alpha\|_{L^2}=1$ and integrating  (\ref{ho11}) with respect to $s$, we have
\begin{eqnarray}\label{ho2}
	\Delta(\phi_{1}^1\circ(\phi_{2}^1)^{-1},\alpha)_{\phi_{2}^1\circ(\psi^1)^{-1}(y)}&=&\int_{0}^1\int_{0}^1 \frac{d}{ds}(K_{1}\sharp\bar{K}_{2})^{t}(\phi^s_{\alpha}(z))dsdt\nonumber\\
	&=&	\int_{0}^1\left [ (K_{1}\sharp\bar{K}_{2})^{t}(\phi^1_{\alpha}(z))- (K_{1}\sharp\bar{K}_{2})^{t}(z)\right ]dt
\end{eqnarray}
Since the function $K_{1}\sharp\bar{K}_{2}$ is compactly supported in $D_{\epsilon}$,
we have
\begin{eqnarray*}
	(K_{1} \sharp \bar{K}_{2})^{t}(\phi^1_{\alpha}(z)) - (K_{1} \sharp \bar{K}_{2})^{t}(z) = 
	\begin{cases}
		(K_{1} \sharp \bar{K}_{2})^{t}(p), & \text{if } z \notin D_{\epsilon} \text{ and } \phi^1_{\alpha}(z) = p \in D_{\epsilon}, \\
		0, & \text{if } z \notin D_{\epsilon} \text{ and } \phi^1_{\alpha}(z) = p \notin D_{\epsilon}, \\
		-(K_{1} \sharp \bar{K}_{2})^{t}(z), & \text{if } z \in D_{\epsilon} \text{ and } \phi^1_{\alpha}(z) = p \notin D_{\epsilon}, \\
		(K_{1} \sharp \bar{K}_{2})^{t}(p) - (K_{1} \sharp \bar{K}_{2})^{t}(z), & \text{if } z \in D_{\epsilon} \text{ and } \phi^1_{\alpha}(z) = p \in D_{\epsilon}.
	\end{cases}
\end{eqnarray*}
For each $z\in D_{\epsilon}$, we have 
\begin{eqnarray}\label{ho3}
	\int_{0}^1 (K_{1}\sharp\bar{K}_{2})_{t}(z)	dt	&=&\int_{0}^1[K_{1}^t-K_{2}^t\circ\phi^t_{K_{2}}\circ(\phi^t_{K_{1}})^{-1}](z)dt\nonumber\\
	&=& \int_{0}^1\left (a(t)+\frac{\epsilon-3CE}{2}\right )h(z)dt\nonumber\\
	&&\quad\qquad-\int_{0}^1 \left (b(t)-\frac{\epsilon-3CE}{2}\right )h\left (\phi^t_{K_{2}}\circ(\phi^t_{K_{1}})^{-1}(z)\right )dt\nonumber\\
	&\leqslant& \int_{0}^1(a(t)-b(t)+\epsilon-3CE)\left [h(p)+h\left (\phi^t_{K_{2}}\circ(\phi^t_{K_{1}})^{-1}(z)\right )\right ]dt\nonumber\\
	&\leqslant& \int_{0}^1 2(a(t)-b(t)+\epsilon-3CE) \quad \text{since $h(x)\leqslant 1 \  \forall z\in D_{\epsilon}$}\nonumber\\
	&\leqslant& \int_{0}^1 2(CE+\epsilon-3CE) dt \quad \text{by definition of $a(t)$ and $b(t)$}\nonumber\\
	&=&2\epsilon-4CE.
\end{eqnarray}
Therefore, (\ref{ho2}) becomes
\begin{eqnarray}\label{ho4}
	\Delta(\phi_{1}^1\circ(\phi_{2}^1)^{-1},\alpha)_{\phi_{2}^1\circ(\psi^1)^{-1}(y)}
	&\leqslant&2\epsilon-4CE.
\end{eqnarray}
Putting (\ref{ho4}) in (\ref{qs0})  we have
\begin{eqnarray}\label{qs1}
	\Delta(\phi_{1}^1\circ(\psi^1)^{-1},\alpha)_{y}\leqslant 	\Delta(\phi_{2}^1\circ(\psi^1)^{-1},\alpha)_{y}+ 2\epsilon-4CE
\end{eqnarray}
Therefore, either $	\Delta(\phi_{2}^1\circ(\psi^1)^{-1},\alpha)_{y}\geqslant2CE-\epsilon$ or $\Delta(\phi_{1}^1\circ(\psi^1)^{-1},\alpha)_{y}\leqslant\epsilon-2CE$. That is either $ \phi_{2}^1$ or $\phi_{1}^1$ satisfies item $(3)$ . From Proposition \ref{pro01}, we have 
\begin{eqnarray*}
	\bar{d}(\Phi_{1},\Phi)=\bar{d}(\phi_{K_{2}}\circ\Phi, \Phi)\leqslant C^\Phi\bar{d}(\phi_{K_{2}},id),
\end{eqnarray*}
since $K_{2}$ is compactly supported in $D_{\epsilon}$, for an $\epsilon$ small enough, $\bar{d}(\Phi_{1},\Phi)$ can be made as small as we want . 	For item $ 2.$, we have that 
\begin{eqnarray*}
	l^{(1,\infty)}_{ \mathcal{S}}(\Phi_{1})\leqslant l^{(1,\infty)}_{ \mathcal{S}}(\phi_{K_{1}})+l^{(1,\infty)}_{ \mathcal{S}}(\Phi)&\leqslant &\int_{0}^1 \left (a(t)+\frac{\epsilon-3CE}{2}\right)osc(h)dt +E \\
	&\leqslant& \left (osc\left (\mathcal{F}^{\alpha}_{\Phi}(t)(\cdot)\right )+\frac{\epsilon-3CE}{2}\right ) osc(h)+E\\
	&\leqslant&\left (CE+\frac{\epsilon-3CE}{2}\right )osc(h)+E\\
	&\leqslant&\left (\frac{\epsilon-CE}{2}\right )\frac{\epsilon}{CE+\epsilon} +E \ \leqslant\  E+\epsilon.
\end{eqnarray*}
\end{proof}

\begin{remark}$ $
\begin{enumerate}
	\item 	The first de Rham cohomology group $H^1(M,\mathbb{R})$, is a finite dimensional vector space over the separable field $\mathbb{R}$. Hence, $H^1(M,\mathbb{R})$ is separable w.r.t the $L^2-$ norm ($\|\cdot\|_{L^2}$)
	\item $\mathbb{H}_{\mathcal S}(M) $ is separable since it is isomorphic to the separable  space  $H^1(M,\mathbb{R})$. Hence the space of paths  $\mathcal{P}\mathbb{H}_{\mathcal S}(M)$ from $[0,1]$ to $\mathbb{H}_{\mathcal S}(M)$ is separable w.r.t the norm $\displaystyle\int_{0}^1\|\cdot\|_{L^2}dt$
	\item The space of normalised functions $C^\infty_{0}(M\times[0,1],\R)$  is separable when endowed with the $L^{(1,\infty)}-$norm. Therefore, $\mathfrak{T}(M,\omega,\mathcal{S})= C^\infty_{0}(M\times[0,1],\R)\times \mathcal{P}\mathbb{H}_{\mathcal S}(M)$ is separable w.r.t \ $l^{(1,\infty)}_{\mathcal{S}}$ 
\end{enumerate}
\end{remark}
\begin{proof}[Proof of Theorem \ref{bg}]
Let $E>0$ and $(U,\mathcal{H})\in \mathfrak{T}(M,\omega,\mathcal{S})$ such that 
$D^{1}_{\lambda,\mathcal{S}}((U,\mathcal{H},), (0,0))<E$. 
Consider a dense sequence $(V_{j},\mathcal{K}_{j})_{j}\subseteq \mathfrak{T}(M,\omega,\mathcal{S})$, then the  corresponding sequence of time-one maps $\left (\phi^1_{(V_{j},\mathcal{K}_{j})}\right )_{j}\subseteq G_{\omega}(M)$ is dense w.r.t the Hofer-like topology.	Construct inductively a sequence $(U_0,\mathcal{H}_0), (U_1,\mathcal{H}_1), \cdots $  in $\mathfrak{T}(M,\omega,\mathcal{S})$ : by  setting  $(U_0,\mathcal{H}_0)=(U,\mathcal{H})$,  for each $k\geqslant1$
Lemma \ref{lbg} provides us with  $(U_k,\mathcal{H}_k)\in \mathfrak{T}(M,\omega,\mathcal{S})$ such that
\begin{enumerate}
	\item[(a)] $	l^{(1,\infty)}_{ \mathcal{S}}(\phi_{(U_k,\mathcal{H}_k)})<E$  for each $k$,
	\item[(b)] $\bar{d}\left (\phi_{(U_{k-1},\mathcal{H}_{k-1})},\phi_{(U_{k-1},\mathcal{H}_{k-1})}\right )<\frac{\epsilon}{k}$.
	\item[(c)]   for any $\alpha\in \mathcal{Z}^1(M))\smallsetminus\{0\}$,   $\Delta\left (\phi_{(U_{k},\mathcal{H}_{k})}^1\circ \left (\phi_{(V_{j},\mathcal{K}_{j})}^1\right )^{-1},\alpha\right )>CE-\frac{1}{j}$, for $1\leqslant k \leqslant j$.
\end{enumerate}
Property $(b)$ above implies that $\left (\phi_{(U_{k},\mathcal{H}_{k})}\right )_{k}$,converges uniformly to a continuous path $\{\psi^t\}$ of homeomorphisms of $M$. Property $(a)$ and $(b)$ together implies item $1$ of Theorem \ref{bg}. For item $2$,
let $(\psi_{k})\subseteq G_{\omega}(M)$, such that $\psi_{k}\xrightarrow{C^0}\psi^1$. Given $\delta>0$, there exists infinitely many $j$ for which $d_{HL}\left (\psi^1,\phi_{(V_{j},\mathcal{K}_{j})}^1\right )<\delta$.
For such $j$ and for every $k$, there exists a constant $C$ from Theorem \ref{bd} such that 
\begin{eqnarray}\label{l}
	C d_{HL}(\psi_{k},\psi^1)\geqslant  C d_{HL}\left (\psi_{k},\phi_{(V_{j},\mathcal{K}_{j})}^1\right )-C\delta&\geqslant& \left \|\psi_{k}\circ \left (\phi_{(V_{j},\mathcal{K}_{j})}^1\right )^{-1}\right \|^\infty -C\delta\nonumber\\
	&\geqslant& \left |\Delta\left (\psi_{k}\circ \left (\phi_{(V_{j},\mathcal{K}_{j})}^1\right )^{-1},\alpha\right )\right |-C\delta.\qquad \quad
\end{eqnarray} 
Since $\phi_{(U_{k},\mathcal{H}_{k})}^1\xrightarrow{C^0}\psi^1\xleftarrow{C^0}\psi_{k} $,  and $\Delta(\cdot,\alpha)$ is continuous w.r.t the $C^0$ topology, we have 
\begin{eqnarray}\label{l1}
	\lim\limits_{k\longrightarrow\infty}\left |\Delta\left (\phi_{(U_{k},\mathcal{H}_{k})}^1\circ \left (\phi_{(V_{j},\mathcal{K}_{j})}^1\right )^{-1},\alpha\right )\right |&=&
	\left |\Delta\left (\psi^1\circ \left (\phi_{(V_{j},\mathcal{K}_{j})}^1\right )^{-1},\alpha \right )\right |\nonumber\\
	&=& \lim\limits_{k\rightarrow\infty} \left |\Delta\left (\psi_{k}\circ \left (\phi_{(V_{j},\mathcal{K}_{j})}^1\right )^{-1},\alpha \right )\right |\qquad\qquad.
\end{eqnarray}
Combining (\ref{l})  (\ref{l1}) with property (c) one gets 
\begin{eqnarray*}
	C \liminf\limits_{k\rightarrow\infty}d_{HL}(\psi_{k},\psi^1)&\geqslant &\lim\limits_{k\rightarrow\infty} \left |\Delta\left (\psi_{k}\circ \left (\phi_{(V_{j},\mathcal{K}_{j})}^1\right )^{-1},\alpha \right )\right | -C\delta\\
	&= &\lim\limits_{k\longrightarrow\infty}\left |\Delta\left (\phi_{(U_{k},\mathcal{H}_{k})}^1\circ \left (\phi_{(V_{j},\mathcal{K}_{j})}^1\right )^{-1},\alpha\right )\right | -C\delta	\\
	&\geqslant& \lim\limits_{k\rightarrow\infty} \left (CE-\frac{1}{j}- C\delta\right ).
\end{eqnarray*}
Therefore, $\liminf\limits_{k\rightarrow\infty}d_{HL}(\psi_{k},\psi^1)\geqslant E-\delta$, for every $\delta>0$.
\end{proof}

\begin{center}
	{\bf Acknowledgments :}
\end{center}

The first author  acknowledges support of the  CEA-SMIA (IDA N° 6509-BJ et D5320), the Institut Henri Poincaré (UAR 839 CNRS-Sorbonne Université), and LabEx CARMIN (ANR-10-LABX-59-01).


\begin{thebibliography}{}
	\bibitem{Banyaga78}
	A. Banyaga, {\it The Structure of classical diffeomorphisms groups}, 
	Mathematics and its applications {\bf 400}. Kluwer Academic Publisher's Group, Dordrescht, The Netherlands (1997).
	
	\bibitem{ba10a} A. Banyaga, A Hofer-like metric on the group of symplectic diffeomorphisms, Contempt.
	Math. Amer. Math. Soc. RI. {\bf 512} (2010), $1-24$.
	
	\bibitem{Ba08} A. Banyaga, {\it On the group of strong symplectic homeomorphisms},
	C. R. Math. Acad. Sci. Paris,  {\bf 346}(15-16)(2008), $867-872$.
	
	\bibitem{Ba10} A. Banyaga. {\it On The Group of Strong Symplectic Homeomorphisms},
	CUBO A Mathematical Journal, {\bf12}(03) (2010) , $49-69$.
	
	
	\bibitem{bu}L. Buhovsky, {\it On two remarkable groups of area-preserving homeomorphisms}, Journal of Mathematical Physics, Analysis, Geometry,{\bf19}(2) (2023), $339-373.$
	

	\bibitem{cr2} D. Cristofaro-Gardiner, V. Humilière, and S. Seyfaddini. PFH spectral invariants on
	the two-sphere and the large scale geometry of Hofer’s metric. Journal of the European Mathematical Society,(2021): n. pag.
	
	\bibitem{cr3} D. Cristofaro-Gardiner, V. Humilière, C. Y. Mak, S. Seyfaddini, and I. Smith. Quantitative Heegaard Floer cohomology and the Calabi invariant. Forum of Mathematics, Pi (2022), 10, E27 $1-59$. doi:10.1017/fmp.2022.18
	
	\bibitem{Elia}Y. Eliashberg, {\it A theorem on the structure of wave fronts and applications in symplectic topology}, Funct. Anal. Appl. 21 $(1987)$, $227-232.$
	
	\bibitem{gro}M. Gromov, {\it Pseudo-holomorphic curves in symplectic manifolds,}
	Invent. Math. 82 $(1985)$, $307-347$.
	
	\bibitem{Oh-M07}
	Y-G. Oh and S. M\"{u}ller,  {\it The group of Hamiltonian homeomorphisms and $C^0$-symplectic topology,}
	J. Symp. Geometry {\bf 5}$(2007)$, $167-225$.
	
	
	\bibitem{Tc18} S. Tchuiaga, {\it On symplectic dynamics,} Differ. Geom. Appl. {\bf 61}(2018), $170-196$.
	
	\bibitem{tc21} S. Tchuiaga, {\it Hofer-like geometry and flux theory ,} J. Dyn. Syst. Geom. Theories 19 (2) (2021) $227-270,$  https://doi.org/10.1080/172603X.2021.2011110.
	
		
	\bibitem{TM}
	S. Tchuiaga, C. Madengko {\it On the simplicity of the group of strong hamiltonian homeomorphisms"} to appear
	
	\bibitem{Tch-al}
	S. Tchuiaga, F. Houenou, C. Madengko ,
	A. Nguedakumana  {\it $C^0-$Transport of flux geometry} Journal of Topology and its Applications 322 (2022) 108301
\end{thebibliography}
\end{document}